\documentclass[12pt]{amsart}
\usepackage{color}
\usepackage{a4wide}
\usepackage{verbatim}

\usepackage{eucal}

\begin{document}
\newtheorem{theorem}{Theorem}
\theoremstyle{definition}
\newtheorem{remark}[theorem]{Remark}

\def\T{\mathrm{T}}
\def\L{\mathcal{L}}
\def\d{\mathrm{d}}
\def\O{\Omega}
\def\a{\alpha}
\def\b{\beta}
\def\g{\gamma}
\def\la{\langle}
\def\ra{\rangle}
\def\tr{{\rm tr}}
\def\id{{\rm Id}}
\newcommand{\be}{\begin{equation}}
\newcommand{\ee}{\end{equation}}
\def\beq{\begin{eqnarray*}}
\def\eeq{\end{eqnarray*}}
\def\p{\psi}
\def\.{{\cdot}}

\title{On pluricanonical locally conformally K\"ahler manifolds}
\author{Andrei Moroianu} 
\address{Laboratoire de Math\'ematiques de Versailles, UVSQ, CNRS, Universit\'e Paris-Saclay, 78035 Versailles, France }
\email{andrei.moroianu@math.cnrs.fr}
\author{Sergiu Moroianu}
\address{Institutul de Matematic\u{a} al Academiei Rom\^{a}ne\\
P.O. Box 1-764\\RO-014700
Bu\-cha\-rest, Romania}
\email{moroianu@alum.mit.edu}
\thanks{This work was partially supported by the CNCS grant PN-II-RU-TE-2012-3-0492}
\begin{abstract}
We prove that compact pluricanonical locally conformally K\"ahler manifolds have parallel Lee form. 
\end{abstract}

\maketitle

\section{Introduction}

A locally conformally K\"ahler (lcK) manifold is a complex manifold $(M,J)$ together with a Hermitian metric $g$ 
which is conformal to a K\"ahler metric in the neighbourhood of every point. The logarithmic differentials of the conformal factors glue up to a globally defined
closed 1-form $\theta$, called the {\em Lee form}, 
such that the fundamental 2-form $\Omega:=g(J\cdot,\cdot)$ satisfies 
\begin{equation}\label{do}
\d\Omega=\theta\wedge\Omega.
\end{equation}

If $\theta$ vanishes identically, the manifold $(M,g,J)$ is K\"ahler. We will implicitly assume in the whole paper that the lcK structure is proper, i.e. that $\theta$ is not identically zero.

When $\theta$ is parallel with respect to the Levi-Civita connection $\nabla$ of $g$, the lcK manifold $(M,J,g)$ is called \emph{Vaisman}. 

G.\ Kokarev introduced in the context of harmonic maps \cite{K} the seemingly larger class of \emph{pluricanonical lcK manifolds}, 
defined as those lcK manifolds $(M,g,J)$ satisfying 
\begin{equation}\label{defp0}
(\nabla\theta)^{1,1}=0.
\end{equation}

In their recent preprint \cite{OV3}, L. Ornea and M. Verbitsky announce the proof of the following result:
\begin{theorem}\label{main}
Every compact pluricanonical lcK manifold $(M,J,g)$ is Vaisman.
\end{theorem}
The arguments given in \cite{OV3} are 
based on an impressive amount of previous results by numerous authors. 
Among these we mention: the classification of complex surfaces by Kodaira, the classification of complex surfaces of K\"ahler rank $1$ by 
Chiose and Toma \cite{CT} and Brunella \cite{Br}, some results by M.\ Kato concerning subvarieties of Hopf manifolds \cite{Ka},
the classification of surfaces carrying Vaisman metrics by Belgun \cite{Be}, as well as several previous results by Ornea, Kamishima  \cite{KO} and Ornea, Verbitsky, \cite{OV1}, \cite{OV2}. 

As a matter of fact, while were not able to follow their arguments in detail, we discovered instead that Theorem \ref{main} 
can be proved in a more direct way. Our idea is based on the observation that on a pluricanonical manifold, the flows of the metric duals 
$\xi$ and $J\xi$ of $\theta$ and $J\theta$ commute, and this eventually shows that some partial Laplacian (in the analists sense) of the square norm of $\nabla\xi$ is larger than or equal to the square norm of $\nabla\xi\circ \nabla\xi$. By looking at a point where $|\nabla\xi|^2$ is maximal, this shows that $\nabla\xi\equiv 0$. 
The details of the proof are given in Section 3. 

In Section 4 we show, with similar arguments, that for every complete lcK manifold $M$ which is pluricanonical but not Vaisman, there exists a holomorphic and isometric injective immersion $\Phi:\Sigma\to M$ with vanishing second fundamental form, where $\Sigma$ is a Riemannian surface isometric 
either to the Euclidean plane or to a flat cylinder. 

In the final section we explain that due to Theorem 1, Kokarev's strong rigidity \cite[Theorem 5.1]{K} is essentially an empty result. This fact was pointed out by the anonymous referee, whom we also thank for other useful suggestions.

{\em Bibliographical remark.} After the first version of this note was made public, the manuscript \cite{OV3} was withdrawn and replaced with two other preprints \cite{OV4}, \cite{OV5}, where the proof of Theorem \ref{main} is no longer provided, but referred instead to our present work. We acknowledge, however, the decisive influence of \cite{OV3}, which constituted the original impetus for our study.

\section{Preliminaries on lcK metrics}

Whenever a Riemannian metric $g$ is given on a Riemannian manifold, it induces the so-called {\em musical isomorphisms} $\sharp:\T^*M\to \T M$ and $\flat:\T M\to\T^*M$, parallel with respect to the Levi-Civita connection of $g$ and inverse to each other, defined by $g(\alpha^\sharp,\cdot):=\alpha$ and $X^\flat:=g(X,\cdot)$ for every $\alpha\in\T^*M$ and $X\in \T M$.

Assume from now on that $(M,g,J,\theta)$ is an lcK manifold. It is well known that on Hermitian manifolds, the exterior derivative of the fundamental $2$-form $\Omega:=g(J\cdot,\cdot)$ determines its covariant derivative. The formula for the covariant derivative of $J$ determined by \eqref{do} is (see e.g.\ \cite{M}):
\begin{equation} \label{lck} 
\nabla_X J=\frac12\left(X\wedge J\theta+JX\wedge\theta\right),\qquad\forall X\in\T M,
\end{equation}
where if $\alpha$ is a 1-form, $J\alpha$ denotes the 1-form defined by $(J\alpha)(Y):=-\alpha(JY)$ for all tangent vectors $Y$,
and $X\wedge\alpha$ is the endomorphism of the tangent bundle defined by \[(X\wedge\alpha)(Y):=g(X,Y)\alpha^\sharp-\alpha(Y)X.\]
Note that in \cite{M}, a different normalization is used in the definition of the Lee form \eqref{do}, which introduces a factor $\frac12$ in \eqref{lck}, compared to the corresponding formula in \cite{M}.

Let $\xi:=\theta^\sharp$ denote the metric dual 
of the Lee form $\theta$.
Consider the bilinear form $\sigma:=\nabla\theta$ and the associated endomorphism $S:=\nabla\xi$. They are related by the formula $\sigma(\cdot,\cdot)=g(S\cdot,\cdot)$. Since the Lee form $\theta$ is closed, we get for every vector fields $X,Y$ on $M$:
\[0=\d\theta(X,Y)=(\nabla_X\theta)(Y)-(\nabla_Y\theta)(X)=\sigma(X,Y)-\sigma(Y,X),\]
thus showing that $\sigma$ is a symmetric bilinear form, and correspondingly $S$ is a symmetric endomorphism with respect to the metric $g$.

The pluricanonicality condition \eqref{defp0} is equivalent to 
\[0=\nabla\theta(X,Y)+\nabla\theta(JX,JY)=g(SX,Y)+g(SJX,JY)=g(SX-JSJX,Y),\]
for every vector fields $X,Y$ on $M$. We thus see that an lcK structure is pluricanonical if and only if the tensor $S:=\nabla(\theta^\sharp)$ satisfies $S=JSJ$, or equivalently 
\begin{equation}\label{defp}
SJ=-JS.
\end{equation}

\section{Proof of Theorem \ref{main}}

Assume from now on that $(M,g,J,\theta)$ is a pluricanonical lcK manifold.
We need to show that, under the compactness assumption, 
the relation \eqref{defp} implies the vanishing of $S$. 
From the definition of $S$, together with \eqref{lck}, we have 
\begin{equation}\label{nt1}
\nabla_X\xi=SX,\qquad \nabla_X(J\xi)=JSX+\frac12(\theta(X)J\xi+\theta(JX)\xi-|\theta|^2JX),
\end{equation}
which by lowering the indices also reads
\begin{equation}\label{nt}
\nabla_X\theta=(SX)^\flat,\qquad \nabla_X(J\theta)=(JSX)^\flat+\frac12X\lrcorner(\theta\wedge J\theta-|\theta|^2\Omega).
\end{equation}
By \eqref{defp}, the endomorphism $JS$ is symmetric. From \eqref{nt} we thus get 
\begin{equation}\label{djt} 
\d (J\theta)=\theta\wedge J\theta-|\theta|^2\Omega,
\end{equation}
\begin{equation}\label{lieg}
\L_\xi g=2g(S\cdot,\cdot),\qquad \L_{J\xi}g=2g(JS\cdot,\cdot).
\end{equation}
Taking a further exterior derivative in \eqref{djt} and using \eqref {do} yields 
\[0=\d^2 (J\theta)=-\theta\wedge\d(J\theta)-\d(|\theta |^2)\wedge\Omega-|\theta |^2\d\Omega=-\d(|\theta |^2)\wedge\Omega,\]
whence $|\theta|^2$ is constant on $M$ (this constancy property of pluricanonical metrics was already noticed in \cite{OV3}) . We thus obtain for every tangent vector $X$:
\[0=X(|\xi|^2)=2g(\nabla_X\xi,\xi)=2g(SX,\xi)=2g(S\xi,X),\]
showing that $S\xi=0$ (and therefore also $SJ\xi=0$ from \eqref{defp}). Using \eqref{nt1} we thus get 
$\nabla_{J\xi}\xi=\nabla_\xi(J\xi)=\nabla_{J\xi}(J\xi)=\nabla_\xi\xi=0$, 
and in particular 
\begin{equation}\label{com}
[\xi,J\xi]=0.
\end{equation}
We note for later use that the distribution $\{\xi,J\xi\}$ is integrable, and its integral leaves are totally geodesic.

Cartan's formula shows that on every lcK manifold
\begin{equation}\label{ljo}
\L_{J\xi}\Omega=\d(J\xi\lrcorner\Omega)+J\xi\lrcorner\d\Omega=-\d \theta+J\xi\lrcorner(\theta\wedge\Omega)=0.
\end{equation}
Moreover, on pluricanonical manifolds, equation \eqref{djt} gives
\begin{equation}\label{lo}
\L_\xi\Omega=\d(\xi\lrcorner\Omega)+\xi\lrcorner\d\Omega=\d (J\theta)+\xi\lrcorner(\theta\wedge\Omega)=0.
\end{equation}
From \eqref{lieg} and \eqref{lo} we infer
\begin{equation}\label{lieJ}
\L_\xi J=2JS,\qquad \L_{J\xi}J=-2S.
\end{equation}
We notice that \eqref{com} implies $[\L_\xi,\L_{J\xi}]=\L_{[\xi,J\xi]}=0$, and thus from \eqref{lieJ}:
\begin{equation}\label{lieS}
\L_\xi S=-\frac12\L_\xi\L_{J\xi}J=-\frac12\L_{J\xi}\L_\xi J=-\L_{J\xi}(JS)=2S^2-J\L_{J\xi}S,
\end{equation}
which (after composing with $J$ on the left) also reads 
\begin{equation}\label{lieJS}
\L_{J\xi} S=J\L_{\xi}S-2JS^2.
\end{equation}
Taking a further Lie derivative in \eqref{lieS} and using \eqref{lieJ} yields 
\begin{align*}
\L_{J\xi}\L_\xi S={}&2S\L_{J\xi}S+2(\L_{J\xi}S)S+2S\L_{J\xi}S-J\L_{J\xi}^2S\\
={}&4S\L_{J\xi}S+2(\L_{J\xi}S)S-J\L_{J\xi}^2S.
\end{align*}
Similarly, from \eqref{lieJS} and \eqref{lieJ} we obtain:
\begin{align*}
\L_{\xi}\L_{J\xi} S={}&2JS\L_{\xi}S+J\L_\xi^2 S-4JS^3-2J(\L_{\xi}S)S-2JS\L_{\xi}S\\
={}&J\L_\xi^2 S-4JS^3-2J(\L_{\xi}S)S\\
={}&J\L_\xi^2 S-8JS^3-2(\L_{J\xi}S)S.
\end{align*}
Comparing the last two equations and using $\L_\xi\L_{J\xi}=\L_{J\xi}\L_\xi$ we obtain
\begin{equation}\label{final}
J(\L_\xi^2 S+\L_{J\xi}^2S)=4S\L_{J\xi}S+4(\L_{J\xi}S)S+8JS^3.
\end{equation} 
We compose with $-SJ$ to the left and take the trace in the above equation:
\[\tr(S(\L_\xi^2 S+\L_{J\xi}^2S))=-4\tr(SJS(\L_{J\xi}S))-4\tr(SJ(\L_{J\xi}S)S)+8\tr(S^4)=8\tr(S^4),\]
from the trace identity and the hypothesis $SJ=-JS$. Using this we compute:
\begin{align*}
(\L_\xi^2+\L_{J\xi}^2)(\tr(S^2))={}&\tr\left((\L_\xi^2S)S+2(\L_\xi S)^2+S(\L_\xi^2S)+(\L_{J\xi}^2S)S+2(\L_{J\xi} S)^2+S(\L_{J\xi}^2S)\right)\\
={}&2\tr((\L_\xi S)^2)+2\tr((\L_{J\xi} S)^2)+2\tr\left(S(\L_\xi^2S)+S(\L_{J\xi}^2S)\right)\\
={}&2\tr((\L_\xi S)^2)+2\tr((\L_{J\xi} S)^2)+16\tr(S^4).
\end{align*}
By taking the Lie derivative with respect to $\xi$ of the relation $g(S\cdot,\cdot)=g(\cdot,S\cdot)$ and using \eqref{lieg} we immediately get 
$g(\L_\xi S\cdot,\cdot)=g(\cdot,\L_\xi S\cdot)$, i.e., the endomorphism $\L_\xi S$ is symmetric. Taking now the Lie derivative 
of the relation $SJ+JS=0$ with respect to $\xi$ and using \eqref{lieS} we obtain that $\L_\xi S$ anti-commutes with $J$. 
Finally, \eqref{lieJS} shows that the symmetric part of $\L_{J\xi}S$ is $J\L_\xi S$ and its skew-symmetric part is $-2JS^2$. 
The previous relation thus reads
\begin{align*}
(\L_\xi^2+\L_{J\xi}^2)(\tr(S^2))={}&2\tr((\L_\xi S)^2)+2\tr((\L_{J\xi} S)^2)+16\tr(S^4)\\
={}&2\tr((\L_\xi S)^2)+2\tr((\L_{\xi} S)^2-4S^4)+16\tr(S^4)\\
={}&4\tr((\L_\xi S)^2)+8\tr(S^4).
\end{align*}
We use now the compactness assumption: there exists a point $x_{max}\in M$ where $\tr(S^2)$, the square norm of $S$, attains its supremum. 
At $x_{max}$ the left hand side of the equation 
above is non-positive, while the right hand side is non-negative (since we have seen that $\L_\xi S$ is symmetric). 
We deduce that $\tr(S^4)$ -- and thus $S$ itself -- both vanish at $x_{max}$, so $S$ vanishes identically. 
This is the conclusion of Theorem \ref{main}.

\section{Non-compact pluricanonical manifolds}
Our method of proof extends partially to the case where the pluricanonical manifold 
$(M,J,g)$ is complete but not compact. 
\begin{theorem}
Let $(M,J,g)$ be a complete pluricanonical manifold which is not Vaisman. 
Then there exists a Riemannian surface $\Sigma$ isometric 
either to the Euclidean plane $\mathbb{R}^2$ or to a flat cylinder 
$\mathbb{R}^2/l\mathbb{Z}$ for some radius $l>0$, and a holomorphic and isometric immersion $\Phi:\Sigma\to M$ with vanishing second fundamental form. 
\end{theorem}
\begin{proof}
Each leaf $F$ of the foliation tangent to the totally geodesic distribution $\{\xi,J\xi\}$ 
is totally geodesic and, although not necessarily a submanifold in $M$, is a complete flat surface. 
More precisely, there exists a flat Riemannian surface $\Sigma$ and a holomorphic and isometric injective immersion $\Phi:\Sigma\to M$ with vanishing second fundamental form such that $F=\Phi(\Sigma)$.
The universal cover 
of $\Sigma$ is isometric to the Euclidean plane, hence $\Sigma$
is isomorphic (as K\"ahler manifold) to either $\mathbb{R}^2$, a flat cylinder, or a flat torus.

If $\Sigma$ is compact, the endomorphism $S$ vanishes over $F=\Phi(\Sigma)$ by the same argument as in the last paragraph 
of the proof of Theorem \ref{main}. So if $(M,J,g)$ is not Vaisman, we must have at least one non-compact leaf, hence the conclusion of the theorem. 
\end{proof}
We do not know whether there exist complete non-compact pluricanonical manifolds which are not Vaisman.

\section{Final remark}

Kokarev's original motivation for introducing pluricanonical lcK metrics was an attempt to generalize Siu's strong rigidity \cite{S} to a wider class of manifolds. In \cite[Theorem 5.1]{K} he makes the following statement: 

{\em If $M$ is a compact pluricanonical lcK manifold homotopic to a compact locally Hermitian symmetric space of non-compact type $M'$ whose universal cover has no hyperbolic plane as a factor, then $M$ is biholomorphic to $M'$}.

This statement is in fact empty since by Theorem \ref{main}, every compact pluricanonical lcK manifold is Vaisman, thus its first Betti number is odd \cite{KS}, \cite{V}, whereas the first Betti number of a compact locally Hermitian symmetric space is even.

\end{document}